\documentclass{amsart}

\usepackage{amsmath,amssymb,amsfonts,amscd}
\usepackage[english]{babel}
\usepackage{graphicx}

\usepackage{fancyhdr}

\newtheorem{theorem}{Theorem}[section]
\newtheorem{proposition}[theorem]{Proposition}
\newtheorem{lemma}[theorem]{Lemma}
\newtheorem{corollary}[theorem]{Corollary}

\theoremstyle{definition}
\newtheorem{definition}[theorem]{Definition}
\theoremstyle{remark}
\newtheorem{remark}[theorem]{Remark}
\theoremstyle{remark}
\newtheorem{example}[theorem]{Example}
\theoremstyle{remark}

\newcommand{\supp}{\mathrm{supp}}

\newcommand{\fix}{\mathrm{Fix}}
\newcommand{\sym}{\mathrm{Sym}}

\newcommand{\sub}{\mathrm{Sub}}
\newcommand{\reg}{\mathrm{reg}}
\newcommand{\stab}{\mathrm{Stab}}

\newcommand{\ie}{\text{i.e.\;\,}}

\newcommand{\F}{\mathcal F}
\newcommand{\G}{G^{\mathrm {fix}}}

\begin{document}

\title[Invariant Random Subgroups of Full Groups]{On Invariant Random Subgroups of Block-Diagonal Limits of  Symmetric Groups}

\author{Artem Dudko}
\thanks{The research of A.D. was supported by the by the National Science Centre, Poland,
grant 2016/23/P/ST1/04088 under POLONEZ programme which has received funding from the EU\;\protect\includegraphics[width=.03\linewidth]{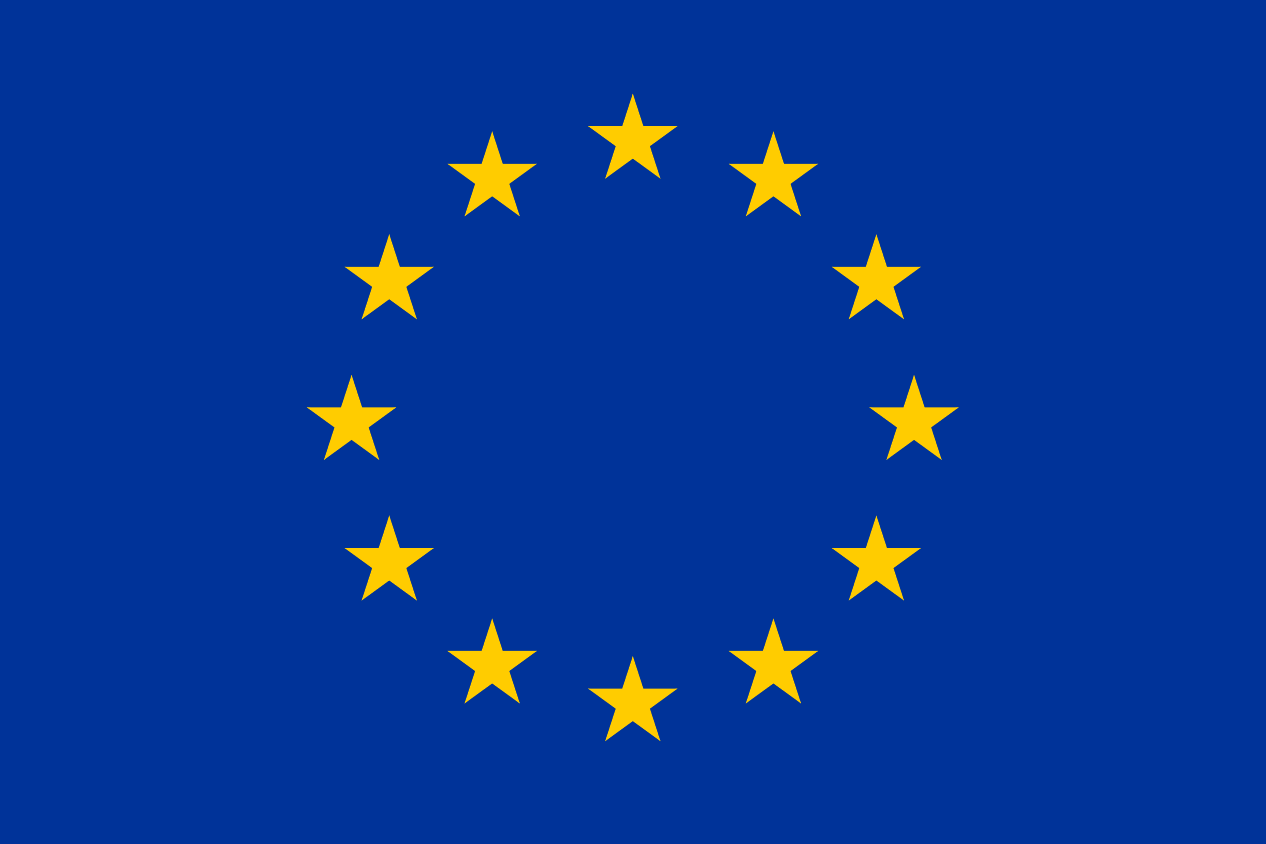} Horizon 2020 research and innovation programme under the MSCA grant agreement No. 665778.}
\address{Institute of Mathematics, Polish Academy of Sciences, Warsaw, Poland}
\email{adudko@impan.pl}

\author{Kostya Medynets}
\thanks{The research of K.M. was supported by NSA YIG H98230258656.}
\address{U.S. Naval Academy, Annapolis, MD, USA}
\email{medynets@usna.edu}


\date{}

\begin{abstract}  We classify the ergodic invariant random subgroups of block-diagonal limits of symmetric groups in the cases when  the groups are simple and the associated dimension groups have finite dimensional state spaces. These block-diagonal limits arise as the transformation groups (full groups) of Bratteli diagrams  that preserve the cofinality of infinite paths in the diagram.  Given a simple full group $G$ admitting only a finite number of ergodic measures on the path-space $X$ of the associated Bratteli digram, we prove that every non-Dirac ergodic invariant random subgroup of $G$ arises as the stabilizer distribution of the diagonal action on $X^n$ for some $n\geq 1$.  As a corollary, we establish that every group character $\chi$ of $G$ has the form $\chi(g) = Prob(g\in K)$, where $K$ is a conjugation-invariant random subgroup of $G$.
\end{abstract}

\maketitle

\section{Introduction}

Let $G$ be a countable discrete group and let $\sub(G)$ be the  space of subgroups of $G$.  The set $\sub(G)$  is compact and zero-dimensional when equipped with the induced topology from  $\{0,1\}^G$. The group $G$ acts on $\sub(G)$ by  conjugation. A Borel probability $G$-invariant measure on $\sub(G)$  is called an {\it invariant random subgroup (IRS) of $G$}. We notice that the Dirac measures supported by the trivial subgroups $\{G\}$ and $\{e\}$ in $\sub(G)$ are invariant random subgroups. More generally, the Dirac IRS's   correspond to normal subgroups of $G$. In some sense, the invariant random subgroups can be regarded as ``generalized'' normal subgroups.

Suppose the group $G$ acts on a measure space $(Y,\nu)$ by measure-preserving transformations. Then the push-forward measure of $\nu$ under the stabilizer map $\stab_G: Y\rightarrow \sub(G)$ given by $\stab_G(y) = \{g\in G : g(y)=y\}$ is an IRS. In fact,
Abert-Glasner-Virag  \cite[Proposition 13]{AbertGlasnerVirag:2014} established that every IRS occurs this way. Creutz-Peterson  \cite[Proposition 3.5]{CreutzPeterson:2017} further refined this result by proving that ergodic invariant random subgroups arise from ergodic actions.

Invariant random subgroups can be used to construct group characters for the group in question. Recall that a character of a group $G$ is a function $f: G\rightarrow \mathbb C$ such that (1) $f(e) = 1$, (2) $f(ab)=f(ba)$ for every $a,b\in G$, and (3) $f$ is positive semidefinite, i.e, the matrix $M = \{f(g_ig_j^{-1})\}_{i,j=1}^n$ is positive semidefinite for any $n\geq 1$ and any family of group elements $g_i\in G$, $i=1,\ldots,n$. The classification of group characters is equivalent to the classification of $II_1$-factor group representations, see, for example, Dudko-Medynets \cite[Section 2.3]{DudkoMedynets:2013}. Given an invariant random subgroup $\varphi$ of $G$, we can associate two characters to $\varphi$: $$\chi_\varphi(g) = \varphi(\{H\in \sub(G) : g\in H\})\mbox{ and }\chi_\varphi'(g) = \varphi(\{H\in \sub(G) : gHg^{-1} = H\}).$$ It is natural to understand for what class of groups (1) the group characters are always of the form $\chi_\varphi$ or $ \chi_\varphi'$ and (2) $\chi_\varphi \equiv \chi_\varphi'$. A possibility of a strong connection between group characters and invariant random subgroups was suggested by Vershik  \cite{Vershik:2010}. Thomas-Tucker-Drob \cite{ThomasTuckerDrob:2014} classified IRSs of diagonal inductive limits of finite symmetric groups in the cases when these groups are simple. Thomas-Tucker-Drob's result along with the classification of characters for such inductive limits obtained by Leinen-Puglisi \cite{LeinenPuglisi:2004}, Dudko \cite{Dudko:2011} and for more general block-diagonal limits by Dudko-Medynets \cite{DudkoMedynets:2013} shows that group characters of diagonal inductive limits of symmetric groups are of the form $\chi_\varphi$. The classification of IRSs and the description of characters via IRSs were later obtained for the class of groups that can be represented as increasing unions of finite alternating groups in Thomas \cite{Thomas:2016}, Thomas-Tucker-Drob \cite{ThomasTuckerDrob:2014,ThomasTuckerDrob:2016}, and Vershik \cite{Vershik:NonfreeActionsSymmetricGroup:2012}.

The main result of the paper is the classification of invariant random subgroups for block-diagonal limits of symmetric groups or, equivalently, for simple AF full groups (see the definition in Section \ref{SectionPreliminaries}) whose associated  Bratteli diagrams admit only finitely many ergodic measures.
One of the simplest examples of  AF full groups can be given as follows. Consider a sequence of natural numbers $\{h_n\}_{n\geq 0}$ such that $h_0 = 1$ and $h_n\geq 2$, $h_n$ divides $h_{n+1}$ for every $n\geq 1$. Set $X_n = \{0,\ldots,h_n-1\}$. Notice that the set $X_{n+1}$ can be represented as a disjoint union of $h_{n+1}/h_n$ copies of $X_n$. Then each element $g$ of the symmetric group $\sym(X_n)$  can be embedded into $\sym(X_{n+1})$ by making it act on each copy of $X_n$ in $X_{n+1}$ as $g$.  Denote by $G$ the inductive limit  of groups $S(X_n)$ under this embedding scheme. The group $G$ is an example of an AF full group. The group  $G$ has a natural continuous action on $X = \prod_{n\geq 1}\{0,\ldots,r_n-1\}$, $r_n = h_{n}/h_{n-1}$, preserving the tails of the sequences. Equivalently,  the space $X$ can be viewed as the path-space of the Bratteli diagram corresponding to the $\{r_n\}$-odometer and the group $G$ becomes the AF full group of the odometer. In this paper, we are interested in the study of full groups of general simple Bratteli diagrams.

The AF full groups arose from the study of orbit equivalence theory of Cantor minimal systems  developed in the series of papers  Herman-Putnam-Skau \cite{HermanPutnamSkau:1992} and Giordano-Putnam-Skau \cite{GiordanoPutnamSkau:1995,GiordanoPutnamSkau:1999} and motivated by applications to the theory of $C^*$-algebras. Giordano-Putnam-Skau  \cite[Corollary 4.11]{GiordanoPutnamSkau:1999} and \cite[Theorem 2.1]{GiordanoPutnamSkau:1995} showed that Cantor minimal $\mathbb Z$-systems are strong orbit equivalent if and only if the associated AF full groups are isomorphic as abstract groups and that the isomorphism of crossed product $C^*$-algebras  is completely characterized by the strong orbit equivalence of underlying dynamical systems. Thus, any information about the algebraic structure of AF full groups can be used to distinguish the associated dynamical systems and the crossed product $C^*$-algebras.

Let $G$ be an AF full group and $B$ be the associated Bratteli diagram. Denote by $X$ the path-space of $B$.  Then the group $G$ acts on $X$ by homeomorphisms by permuting initial segments of the infinite paths. Suppose that the dynamical system $(X,G)$ admits only a finite number of ergodic measures, say,  $\mu_1,\ldots, \mu_k$. Note that $k\geq 1$. For a $k$-tuple $\alpha = (\alpha_1,\ldots,\alpha_k)$, denote by $|\alpha|$ the $1$-norm of $\alpha$. For $\alpha\in \mathbb Z_{\geq 0}^k$, denote by $\mu_\alpha$ the product measure $\mu_1^{\alpha_1}\times \cdots \times \mu_k^{\alpha_k}$ on $X^{|\alpha|}$. Note that $\mu_\alpha$ is $G$-invariant under the diagonal action of $G$ on $X^{|\alpha|}$. By definition, $X^0$ is a singleton and the action of $G$ on $X^0$ is trivial. For $\alpha\in \mathbb Z_{\geq 0}^k$, denote by $\varphi_\alpha$ the stabilizer distribution of $(X^{|\alpha|},\mu_\alpha,G)$. Note that $\varphi_{(0,\ldots,0)} = \delta_{\{G\}}$.   In this paper, under the additional assumption of simplicity of the group $G$ we prove that $\{\varphi_\alpha : \alpha\in \mathbb Z_{\geq 0}^k\}$ are the only non-trivial ergodic IRSs the group $G$ possesses. The proof will rely on the classification of characters for this class of groups established by the authors in \cite{DudkoMedynets:2013}.

\begin{theorem}\label{TheoremMainIntro}
 Let $G$ be a simple AF full group and $X$ be the path-space of the associated Bratteli diagram. Suppose that the dynamical system $(X,G)$ admits only a finite number of ergodic measures $\{\mu_1,\ldots,\mu_k\}$, where  $1\leq k <\infty$.

(1) For every $\alpha \in \mathbb Z_{\geq 0}^k$, the stabilizer distribution $\varphi_\alpha$ of $(X^{|\alpha|},\mu_\alpha,G)$, $\mu_\alpha = \mu_1^{\alpha_1}\times \cdots \times \mu_k^{\alpha_k}$, is an ergodic invariant random subgroup of $G$.

(2) If $\varphi$ is an ergodic invariant random subgroup of $G$, then either $\varphi =\delta_{\{e\}} $ or $\varphi$  is one of the following: $$ \{\varphi_\alpha : \alpha \in \mathbb Z_{\geq 0}^k\}.$$
\end{theorem}

Combining the classification of the group characters established in \cite{DudkoMedynets:2013} with Theorem \ref{TheoremMainIntro}, we obtain the following result.

\begin{theorem}
  Let $G$ be an AF full group satisfying the assumptions of Theorem \ref{TheoremMainIntro}.

(1) For every character $\chi$ of $G$ there exists a unique IRS $\varphi$ such that $$\chi(g) = \varphi(\{H\in \sub(G) : g\in H\})\mbox{ for every }g\in G.$$

(2) For every IRS $\varphi$, $\chi_\varphi(g) = \chi_\varphi'(g)$ for every $g\in G$.
\end{theorem}

We note that the structure of the simplex of invariant measures of any Bratteli diagram  -- and, thus, the structure of the IRSs for the associated full group -- is completely determined by the asymptotic/combinatorial properties of the diagram and has been extensively studied both from the ergodic theory  and operator theory prospectives, see, for example, Bezuglyi-Kwiatkowski-Medynets-Solomyak \cite{BezuglyiKwiatkowskiMedynetsSolomyak:2013} and references therein. We note that the measures on a Bratteli diagram are in one-to-one correspondence with the states on the associated dimension group or, equivalently, $K_0$-group of the associated $AF$-algebra.

In Section \ref{SectionPreliminaries} we give necessary background information on full groups and Bratteli diagrams. In Section \ref{SectionStabilizers} we present several auxiliary results on stabilizer subgroups of full groups and fixed point sets for actions of full groups on Bratteli diagrams. The proofs of the main results are presented in Section \ref{SectionMainResult}.

%
%
%

\section{Preliminaries}\label{SectionPreliminaries}

In this section we collect the notation and basic definitions that
are used throughout the paper. Since the notion of Bratteli
diagrams has been discussed in numerous recent papers, they might be considered as
almost classical nowadays. An interested reader may consult the papers Herman-Putnam-Skau
\cite{HermanPutnamSkau:1992}, Dudko-Medynets \cite{DudkoMedynets:2013},
Bezuglyi-Kwiaktowski-Medynets-Solomyak
\cite{BezuglyiKwiatkowskiMedynetsSolomyak:2013} and references
therein for all details concerning Bratteli diagrams and related dynamical concepts. We only give here some basic definitions in order to fix our
notation.

\begin{definition}\label{Definition_Bratteli_Diagram} A {\it Bratteli diagram} is
an infinite graph $B=(V,E)$ such that the vertex set
$V=\bigcup_{i\geq 0}V_i$ and the edge set $E=\bigcup_{i\geq 1}E_i$
are partitioned into disjoint subsets $V_i$ and $E_i$ such that

(i) $V_0=\{v_0\}$ is a single point;

(ii) $V_i$ and $E_i$ are finite sets;

(iii) there exist a range map $r$ and a source map $s$ from $E$ to
$V$ such that $r(E_i)= V_i$, $s(E_i)= V_{i-1}$, and
$s^{-1}(v)\neq\emptyset$, $r^{-1}(v')\neq\emptyset$ for all $v\in V$
and $v'\in V\setminus V_0$.
\end{definition}

The pair $(V_i,E_i)$ or just $V_i$ is called the $i$-th level of the diagram $B$.
 A finite or infinite sequence of edges $(e_i : e_i\in E_i)$ such
that $r(e_{i})=s(e_{i+1})$ is called a {\it finite} or {\it infinite
path}, respectively. We write $e(v,v')$ to denote a path $e=(e_i,e_{i+1},\ldots,e_j)$ such
that $s(e_i)=v$ and $r(e_j)=v'$. For a Bratteli diagram $B$, we denote
by $X_B$ the set of infinite paths starting at the vertex $v_0$. We
endow $X_B$ with the topology generated by cylinder sets
$U(e_1,\ldots,e_n)=\{x\in X_B : x_i=e_i,\;i=1,\ldots,n\}$, where
$(e_1,\ldots,e_n)$ is a finite path from $B$. Then
$$X_B = \left\{\{e_n\}\in \prod_{i\geq 1}E_i : s(e_{i+1}) = r(e_i)\mbox{ for every }i\geq 1\right\}$$ is a
0-dimensional compact metric space with respect to the product topology.

Given a Bratteli diagram $B$, for every $n\geq 1$ denote by $G_n$ the group of homeomorphisms of $X_B$ that permute only the initial $n$ segments of the infinite paths $\{e_1,\ldots,e_n,e_{n+1}\ldots\}\in X_B.$  For each $n\geq 1$ and each vertex $v\in V_n$, denote by $X_v^{(n)}$ the set of all infinite paths $\{e_1,e_2,\ldots\}$ such that $r(e_n) = v$. Denote by $h_v^{(n)}$ the number of finite paths connecting the root vertex $v_0$ to the vertex $v$ and denote by $E(v_0,v)$ the set of finite paths connecting these vertices.  Set $X_v^{(n)}(\overline e)$, $\overline e\in E(v_0,v)$, to be the set of infinite paths whose first $n$ initial segments coincide with those of $\overline e$. Note that $X_v^{(n)}(\overline e)$ is a clopen set.
Then for every $n\geq 1$ and every $v\in V_n$, we have that $$X =  \bigsqcup_{w\in V_n}X_w^{(n)}  \mbox{ and }X_v^{(n)} = \bigsqcup_{\overline e \in E(v_0,v)}X_v^{(n)}(\overline e). $$

Denote by $G_v^{(n)}$ the subgroup of $G$ whose elements permute only the first $n$ segments of paths from $X_v^{(n)}$. Thus,  the group $G_v^{(n)}$ is isomorphic to the symmetric group on $E(v_0,v)$, that is, $G_v^{(n)} \cong \sym(h_v^{(n)})$.  It follows that $$G_n = \prod_{v\in V_n} G_v^{(n)} \cong \prod_{v\in V_n}\sym(h_v^{(n)}).$$  Set $G_B = \bigcup_{n\geq 1}G_n$. Note that $G_n\subset G_{n+1}$ for every $n\geq 1$ and $G_B$ is a locally finite group.

\begin{definition}
  \label{DefinitionAFFullGroup} Given a Bratteli diagram $B$, the group $G_B$ defined above is called the {\it full group associated to the diagram $B$}. We will simply write $G$ when the diagram $B$ is obvious form the context.
\end{definition}

The following remark reveals the connection between algebraic properties of the full groups and combinatorial properties of the associated Bratteli diagrams. The proofs and the related references can be found in Dudko-Medynets \cite[Section 2.1]{DudkoMedynets:2013}.

\begin{remark} Let $B = (V,E)$ be a Bratteli diagram and $G_B$ be the associated full group.

\begin{enumerate}

\item The dynamical system $(X_B,G_B)$ is  minimal, that is every $G_B$-orbit is dense, if and only if the Bratteli diagram $B$ is {\it simple}, that is, for every $n\geq 1$ there exists $m> n$ such that every vertex in $V_n$ is connected to every vertex in $V_m$.

\item The dynamical system $(X_B,G_B)$ is  minimal if and only if the commutator subgroup of $G_B$ is simple, i.e., has no normal subgroups.

\item The group $G_B$ is simple if and only if for every $n\geq 1$ there exists $m> n$ such that every vertex in $V_n$ is connected to every vertex in $V_m$ and the number of paths between these vertices is even. We refer to Bratteli diagrams with this property as {\it even diagrams}. In this case, the group $G_B$ coincides with its commutator subgroup.
\end{enumerate}
\end{remark}

Fix an even Bratteli diagram $B$. Suppose that the dynamical system $(X_B,G_B)$ admits only a finite number of ergodic measures $\{\mu_1,\ldots,\mu_k\}$. For example, any Bratteli diagrams whose number of vertices per level is uniformly bounded, say, by $K$ cannot admit more than $K$ ergodic measures, see Bezuglyi-Kwiatkowski-Medynets-Solomyak \cite[Proposition 2.13]{BezuglyiKwiatkowskiMedynetsSolomyak:2013}.

\begin{example} Given a Bratteli diagram $B = (V,E)$, the incidence matrix $F_n = (f_{v,w}^{(n)})$, $n\geq 1$, is a $|V_{n+1}|\times |V_n |$ matrix whose entries $f_{v,w}^{(n)}$ are equal to the number of edges between the vertices $v\in V_{n+1}$ and $w\in V_n$.  Consider the Bratteli diagram $B$ given by the sequence of incidence matrices
$$F_n = \left(
 \begin{array}{cc}
1  & 1 \\
n & 1 \\
\end{array}
\right).$$

The diagrammatic representation of $B$ is shown in the following figure.

\begin{figure}[tph!]
\centering
        \includegraphics[totalheight=2.7cm]{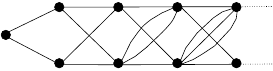}
    \caption{Diagrammatic representation of a Bratteli diagram.}
    \label{fig:verticalcell}
\end{figure}
\end{example}

Let $h_b^{(n)}$  and $h_t^{(n)}$ be the number of edges connecting the root (leftmost)
vertex to the bottom and to the top vertex of the level $n$, respectively.
Then the full group of the diagram $B$ is isomorphic to
$G_B = \bigcup_{n\geq 1}G_n$, where $G_n = \sym(h_b^{(n)})\times \sym(h_t^{(n)})$ and the embedding scheme is described by the diagram $B$. Bezuglyi-Kwiatkowski-Medynets-Solomyak
\cite[Example 5.8]{BezuglyiKwiatkowskiMedynetsSolomyak:2013} showed that the dynamical system $(X_B,G_B)$
is uniquely ergodic and that $$h_b^{(n)} = 2\left(\frac{-i}{\sqrt 2} \right)^{n-2}H_{n-2}(i\sqrt 2),\;n\geq 2,$$
and
$$h_t^{(n)} = nh_b^{(n-1)}+(n-3)h_b^{(n-2)},\;n\geq 3.$$ Here $H_n(x)$ is the $n$-th Hermite polynomial.

\section{Stabilizer Subgroups and Fixed Point sets}\label{SectionStabilizers}

In this section we establish several properties of stabilizer subgroups of full groups and their fixed point sets.

\begin{definition}
  Suppose a group $G$ acts on a space $Z$. Given a group element $g\in G$, the set $\fix_Z(g) = \{z\in Z : g(z) = z\}$ is called the {\it set of fixed points of $g$}. We will simply write $\fix(g)$ when the group action is evident from the context.
\end{definition}

\begin{definition} Let $G$ be the full group of a Bratteli diagram $B$ and $X_B$ be the path-space of $B$.

(1) For a closed set $A\subset X_B$, denote by $ \G(A)$ the subgroup of elements $g\in G$ such that $\fix(g) = \{x\in X_B : g(x) = x\}$ contains the set $A$. 
In other words, if $g\in \G(A)$, then $g(x) = x$ for every $x\in A$. We note that the group $\G(A)$ is supported by the complement of $A$. Observe that $\G(\varnothing)=G$.

(2) For a closed set $A\subset X_B$, denote by $\F(A)$ the set of subgroups $H\in \sub(G)$ such that $gHg^{-1} = H$ for every $g\in \G(A)$.
\end{definition}

\begin{remark}\label{RemarkFFamilyIncreasing} (1) We note that if $A\subset B\subset X$, then $\F(A)\subset \F(B)$.

(2) Let $B = (V,E)$ be an even Bratteli diagram, $X$ the path-space, and $G$ the associated full group. Then for each clopen set $A\subset X$, the group $\G(A)$ is simple, see Dudko-Medynets \cite[Section 2.1]{Dudko:2011}.
\end{remark}

The following lemmas establish several properties of the group $\G(A)$ and the corresponding stabilizer set $\F(A)$.

\begin{lemma}\label{LemmaProperties1} Let $G$ be the full group of a simple
Bratteli diagram and $X$ be the path-space of the diagram. Then for any clopen
sets $A,B\subset X$ with $A\cup B \neq X$, we have that

\begin{enumerate}
\item $\G(A\cap B) = <\G(A),\G(B)>$;

\item $\F(A\cap B) = \F(A)\cap \F(B)$.
\end{enumerate}
\end{lemma}
\begin{proof} We will start by noticing that Statement (2) follows from
Statement (1). Indeed, $H\in \F(A)\cap \F(B)$ if and only if $gHg^{-1} = H$ for
every $g\in \G(A)\cup \G(B)$. Therefore, $H\in \F(A)\cap \F(B)$ if and only if
$gHg^{-1} = H$ for every $g\in \G(A\cap B)$.

To establish Part (1) of the result, we must show that the subgroups
$\G(A)$ and $\G(B)$ generate the group $\G(A\cap B)$. It will be convenient to
work with the {\it local subgroups} instead of $\G(A)$. For a clopen set $C$,
denote by $L(C)$ the set of  elements of $G$ supported by the set $C$.  Note that $\G(C) =
L(X\setminus C)$. Thus, it suffices to prove that  $ L(C\cup
D) = <L(C),L(D)> $ whenever
$C$ and $D$ are clopen sets such that $C\cap D \neq \emptyset$.

First of all, notice that if $g\in L(C)$, then $g(x) = x$ for every $x\in X\setminus C$. In particular, $g(x) = x$ for $x\in X\setminus (C\cup D)$. Hence, $g\in L(C\cup D)$ and $L(C)\subset L(C\cup D)$. Similarly, $L(D)\subset L(C\cup D)$.

Conversely, for a clopen  set $W$, denote by $L_n(W)$ the set of elements of
$G_n$ supported by $W$. Note that $L(W) = \bigcup_{n\geq 1}L_n(W)$. We observe
that for any pair of finite sets $W'$ and $W''$, the symmetrc groups
$\sym(W')$ and $\sym(W'') $ generate $\sym(W'\cup W'')$ whenever $W'\cap W''
\neq \emptyset$. By minimality of $(X,G)$,  for all
$n$ large enough and every $x\in X$, the $G_n$-orbit of $x$ intersects the sets $C\setminus (C\cap D)$, $C\cap D$, and $D\setminus (C\cap D)$. Therefore, $<L_n(C),L_n(D)> =
L_n(C\cup D)$ for all $n$ large enough. It follows that
$L(C\cup D) = <L(C),L(D)> $.
\end{proof}

\begin{lemma}\label{LemmaProperties2}  For any clopen set $A\subset X$ we have
\begin{enumerate}
\item for any $x\in X\setminus A$, the orbit $\G(A)x$ is infinite;
\item for any $H\in\sub(G)\setminus \F(A)$, the orbit $\G(A)H=\{gHg^{-1}:g\in
\G(A)\}$ is infinite.
\end{enumerate}
\end{lemma}
\begin{proof} Since the group $\G(A)$ is infinite and simple, it does not admit non-trivial actions on finite sets.  Therefore, every orbit $\G(A)x$, $x\in X\setminus A$, is infinite. Analogously, we obtain that for every $H\in\sub(G)\setminus \F(A)$, the orbit $\G(A)H$ must be infinite.
\end{proof}

\begin{lemma}\label{LemmaDecrFam} Let $\{C_n\}_{n=1}^\infty$ be a decreasing family of closed sets and $C = \bigcap_{n\geq 1}C_n$. Then $$\bigcap_{n\geq 1}\F(C_n) = \F(C).$$
\end{lemma}
\begin{proof} The inclusion $\bigcap_{n\geq 1}\F(C_n)\supset \F(C)$ follows from Remark \ref{RemarkFFamilyIncreasing}. On the other hand, assume that $K\in\F(C_n)$ for every $n\geq 1$. Then $gKg^{-1}=K$ for every $g\in \bigcup\limits_{n\geq 1}\G(C_n)$. Let $g\in \G(C)$. Note that the set $\fix(g)\supset C$ is clopen. Therefore, by compactness of $X$ we obtain that there exists $n$ such that $C_n\subset \fix(g)$. Thus, $g\in\G(C_n)$. It follows that $\bigcup\limits_{n\geq 1}\G(C_n)=\G(C)$, which shows that $K\in \mathcal F(C)$.
\end{proof}
\begin{lemma}\label{LemmaProperties3} Let $C=\{c_1,\ldots c_m\}\subset X$ be a finite set such that the orbits $Gc_i,$ $i=1,\ldots,m$, are pairwise disjoint. Then $$\F(C) = \bigcup_{Z\subset C} \{\G(Z)\}\cup\{\{e\}\}.$$
\end{lemma}
 \begin{proof} Let $K$ be a subgroup of $G$ such that $g Kg^{-1} = K$ for every  $g\in \G(C)$. First, assume that for every $i=1,\ldots,m$, there exists $k_i\in K$ such that $k_i(c_i) \neq c_i$.
In what follows, it will be convenient to  write  $b^a$ for $aba^{-1}$.   Set $g_m = 1$. Find $g_{m-1}\in \G(C)$ such that $k_{m-1}^{g_{m-1}}k_m^{g_m}(c_l)\neq c_l$ for $l=m-1,m$. Proceeding by induction, we can find a sequence of elements $g_m,g_{m-1},\ldots,g_1$ from $\G(C)$ such that $$k_1^{g_1}\cdots k_m^{g_m}(c_l)\neq c_l\mbox{ for }l=1,\ldots,m.$$

Set $k = k_1^{g_1}\cdots k_m^{g_m}$.  Then $k\in K$ and $k(c_i)\neq c_i$ for every $i=1,\ldots,m$. Additionally conjugating  $k$ by elements of $\G(C)$ and multiplying by $k^{-1}$ if needed, we can show that for every $g\in G$ there exists  $h\in K$ such that $g(C) = h(C)$. Therefore, $g = g_0h$, where $g_0\in \G(C)$. Thus, for every $q\in K$, $$g q g^{-1} = g_0h q h^{-1} g_0^{-1} \in K,$$ which shows that $K$ is a normal subgroup of $G$. By simplicity of $G$, we conclude that $K = G$ or $K = \{e\}$.

Now, if the group $K$ stabilizes some points in $C$, we can find the largest subset $Z\subset C$ with $K\subset \G(Z)$. Using the previous argument and the fact the group $\G(Z)$ is simple,  we prove that $K = \G(Z)$.
\end{proof}

In \cite[Theorem 2.4]{ThomasTuckerDrob:2016} Thomas and Tucker-Drob observed that every ergodic action of a locally finite infinite simple group must be weakly mixing. Their idea was to use the fact that if an action of such a group $G$ is not weakly mixing, then the Koopman representation has a finite-dimensional invariant subspace, which implies that the simple group under consideration admits a finite-dimensional faithful unitary representation. Then by the Jordan-Schur theorem the group $G$ must be virtually Abelian, which contradicts the simplicity of $G$. We would like to note that the class of groups for which ergodicity of an action automatically implies weak mixing also contains certain Lie groups \cite{CreutzPeterson:2017}.

 We recall that one of the equivalent characterizations of weak mixing for a measure-preserving system system $(X,\mu,G)$ is the condition that if $(Y,\nu,G)$ is an ergodic system, then the product system $(X\times Y,\mu\times \nu, G)$ is also ergodic. This leads to the following result.

\begin{proposition}\label{PropositionErgodicityProductMeasures}  Let $G$ be the AF full group of a simple Bratteli diagram $B$. Suppose that $X$ is the path-space of $B$ and
  $\{\nu_1,\ldots,\nu_m\}$ is a collection of ergodic $G$-invariant measures on $X$, not necessarily distinct.  Then the diagonal action of $G$ on $(X^m,\nu)$, $\nu = \nu_1\times \cdots \times \nu_m$, is ergodic.
\end{proposition}
\begin{proof} Note that the commutator subgroup $G'$ of $G$ is simple. We notice that the actions of $G'$ and $G$ on $(X^m,\nu)$ give rise to the same orbit equivalence relations. Therefore, they are either both ergodic or both non-ergodic. Since the action of $G$ on $(X,\nu_i)$ is ergodic for each $i$ it follows that the action of $G'$ on $(X,\nu_i)$ is ergodic and, moreover, weakly mixing for each $i$. Using Thomas and Tucker-Drob's result, we obtain that the action of $G'$ on $(X^m,\nu)$ is ergodic.
\end{proof}

\begin{proposition}\label{PropositionErgodicAverage} Let $H=\bigcup_{n\geq 1} H_n$ be  an
increasing sequence of finite groups acting on a
measure space $(Y,\nu)$. Assume that the orbit $H\cdot y$ of
$\nu$-almost every $y\in Y$ is infinite.  Then
\begin{equation}\lim\limits_{n\to\infty}\frac{1}{|H_n|}\sum\limits_{h\in
H_n}\nu(\fix_Y(h)) = 0\end{equation} and the pre-limit sequence is monotone.
\end{proposition}
\begin{proof} Consider the space $$R = \{(y, h\cdot y) :y\in Y,\;
h\in H\}\subset X\times X.$$ Let $\overline \nu$ be the ``counting measure'' on $R$, that is,
for a measurable set $A\subset R$ we have that $$\overline \nu(A) = \int_{Y}
\textrm{card}(A_y)d\nu(y).$$ Here, $A_y = A\cap \{(y,y') \in R\}$. Set
$\mathcal H = L^2(R,\overline \nu)$ and $\pi(h) = f(h^{-1} x,y)$ for every
$h\in H$ and $f\in \mathcal H$. Then $\pi$ is a unitary representation of
$H$ on the Hilbert space $\mathcal H$. The above construction of the representation $\pi$ is often referred to as the \emph{groupoid construction}.

Set
$$P_n=\frac{1}{|H_n|}\sum\limits_{h\in H_n}\pi(h).$$ Note that $P_n^* = P_n$ and
$P_n^2 = P_n$, that is, $P_n$ is an orthogonal projection. Furthermore, since $H_n\subset H_m$ for $m\geq n>0$, we get that $P_nP_m
= P_m P_n = P_m$ whenever $m\geq n$. In other words, $\{P_n\}_{n\geq 1}$ is a
decreasing sequence of projectors.

It follows that the sequence $\{P_n\}_{n=1}^\infty$ converges in the strong operator
topology to some orthogonal projection $P$, see, for example,  \cite[Corollary 2.5.7]{KadisonRingrose:VolI}.  To establish the result, it suffices to prove that $P = 0$.

Notice that for every $x\in \mathcal H$ and $h\in H$ and for all $n$ large enough, we have that
$$ \|\pi(h)Px- P_nx\|  =  \|\pi(h)Px-\pi(h)P_nx\| =\|Px-P_nx\|\to 0\mbox{ as }n\to\infty.$$
 Therefore, $\pi(h)P=P$ for all $h\in H$.  Assume there exists $\eta\in \mathcal H,\eta\neq 0$, such that $P\eta=\eta$. Then,
$\pi(h)\eta=\eta$ for every $h\in H$. Recall that
$$(\pi(h)\eta)(x,y)=\eta(h^{-1}x,y).$$ In particular, we obtain that for $\overline \nu$-almost all $(x,y)\in R$ if $\eta(x,y)\neq 0$, then $\eta(y,y)\neq 0$ and  there exist infinitely many $z$ such that $\eta(z,y)=\eta(y,y)$. Fix  $Z\subset Y$, $\nu(Z)>0$, such that $\eta(z,z)>0$ for every $z\in Z$.    Denote by $[H]$ the group of $\nu$-preserving transformations of $Y$ that preserve the equivalence relation $R$, the so-called {\it full group of $R$}. Using the standard arguments, we can find an infinite sequence of elements $\{h_n\}\subset [H]$ such that $h_n^{-1}z\neq h_m^{-1}z$ for almost every $z\in Z$ and $n\neq m$, see, for example, \cite[Theorem 3.5]{Kechris:BookGlobalAspects}. Set $Z_n = $ $\{(h_n^{-1}z,z) | z\in Z\}$. Note that $\bar\nu(Z_n) = \nu(Z)>0$ and $\bar\nu(Z_n\cap Z_m) = 0$ whenever $n\geq m$.   It follows from the definition of $\mathcal H$ that
\begin{eqnarray*}
  \|\eta\|_2^2  & \geq  & \sum_{n\geq 1} \int_{Z_n}|\eta(x,y)|^2d\bar\nu(x,y) \\
   &=& \sum_{n\geq 1} \int_{Z}|\eta(h_n^{-1}z,z)|^2d\nu(z) \\
   &=& \sum_{n\geq 1} \int_{Z}|\eta(z,z)|^2d\nu(z) \\
   & = & \infty.
\end{eqnarray*}
This contradiction shows that $P=0$.
\end{proof}

The following result is an immediate application of Proposition \ref{PropositionErgodicAverage} and covers the case of locally finite groups $G = \bigcup_{n\geq 1}G_n$ coming from Bratteli diagrams.

\begin{corollary}\label{corollaryLimitFixedPoints} Let $G$ be an AF-full group with the path-space $X$ and let $A$ be a clopen subset of $X$.
\begin{enumerate}
\item For every $k\geq 0$ and every  $\bar\mu$ on $X^k$ invariant under the diagonal action of $G$, we have that
$$\lim_{n\to\infty}\frac{1}{|\G_n(A)|}\sum\limits_{g\in \G_n(A)}\bar\mu(\fix_{X^k}(g)) = \bar\mu(A^k).$$
\item For every IRS $\varphi$, we have that
$$\lim_{n\to\infty}\frac{1}{|\G_n(A)|}\sum\limits_{g\in \G_n(A)}\varphi(\fix_{\sub(G)}(g)) = \varphi(\F(A)).$$
\end{enumerate}
  The convergence in each case is monotone.
\end{corollary}
\begin{proof} Note that for every clopen set $A\subset X$, the sets $A^k$ and $X^k\setminus A^k$ are invariant under the action of the subgroup $\G(A)$ and that every $\G(A)$-orbit in $X^k\setminus A^k$ is infinite, see Lemma \ref{LemmaProperties2}. Thus, using Proposition \ref{PropositionErgodicAverage} we obtain that
\begin{multline*}\frac{1}{|G_n(A)|}\sum\limits_{g\in G_n(A)}\bar\mu(\fix_{X^k}(g)) = \frac{1}{|G_n(A)|}\sum\limits_{g\in G_n(A)}\big(\bar\mu(\fix_{A^k}(g))+\bar\mu(\fix_{X^k\setminus A^k}(g))\big)
\\
= \bar\mu(A^k) + \frac{1}{|G_n(A)|}\sum\limits_{g\in G_n(A)}\bar\mu(\fix_{X^k\setminus A^k}(g)) \to \bar\mu(A^k). \end{multline*}
The second assertion can be  established using similar arguments. We leave the details to the reader.
\end{proof}

\section{Proof of the main result}\label{SectionMainResult}

%
%
%
This section is devoted to the proof of Theorem \ref{TheoremMainIntro}.  Throughout the section we assume that $G$ is a simple AF full group and $X$ is the path-space of the associated Bratteli diagram. Additionally, we  assume that the dynamical system $(X,G)$ admits only a finite number of ergodic measures $\{\mu_1,\ldots,\mu_k\}$, where  $1\leq k <\infty$.

 Set $\Omega = \sub(G)$. First of all, we notice that according to Proposition \ref{PropositionErgodicityProductMeasures} for every $\alpha\in \mathbb Z_{\geq 0}^k$, the measure $\mu_\alpha$ on $X^{|\alpha|}$ is ergodic under the diagonal action of $G$. Therefore, the corresponding stabilizer distribution $\varphi_\alpha$ on $\Omega$  is an ergodic IRS.

 Conversely, fix an ergodic IRS $\varphi$ of $G$.
%
If $\varphi$ has atoms, then by ergodicity, the measure $\varphi$ is supported by a normal subgroup of $G$, which, in view of simplicity of $G$, implies that $\varphi = \delta_{\{e\}}$ or $\varphi = \delta_{G}$. In what follows we  assume that the measure $\varphi$ has no atoms.
For every $g\in G$ define $\chi(g)=\varphi(\fix_\Omega(g))$. Then  the function $\chi: G\rightarrow \mathbb C$ is a character on $G$, see, for example, \cite{Vershik:2010}. Using the description of indecomposable characters for $G$ obtained by the authors in  \cite{DudkoMedynets:2013}, we conclude that there exists a collection of nonnegative real numbers $c_\alpha,\alpha\in\mathbb Z_{\geq 0}^k$, that adds up to 1 and such that
$$\chi=\sum\limits_{\alpha\in \mathbb Z_{\geq 0}^k}c_\alpha\chi_\alpha+c_\infty\chi_\reg, $$ where
$$\chi_\alpha(g)=\mu_\alpha(\fix_{X^{|\alpha|}}(g))=
\prod\limits_{i=1}^k\mu_i(\fix_X(g))^{\alpha_i}$$ is the character corresponding to $\varphi_\alpha$ and $\chi_\reg$ is the regular character of $G$.

\begin{lemma}\label{LemmaCinfty}  There exists $\alpha\in \mathbb Z^k_{\geq 0}$ such that  $c_\alpha>0$.\end{lemma}
 \begin{proof} Assume that $c_\infty=1$. Given an element $g\in G$, set $\mathcal H_g=\{H<G: g\in H\}$. Note that $\mathcal H_g\subset \fix_\Omega(g)$. Therefore,
 $$\varphi(\mathcal H_g) \leq \varphi(\fix_\Omega(g)) = \chi_\reg(g) = 0,$$ whenever $g\neq e$.  Since $\cup_{g\neq e} \mathcal H_g=\Omega\setminus \{e\}$, we conclude that $\varphi$ is supported on $\{e\}$. Thus, $\varphi= \delta_{\{e\}}$ is an atomic measure, which is a contradiction. \end{proof}

Set $$r = \min\{k:c_\beta>0\;\;\text{for some}\;\;\beta \in \mathbb Z^k_{\geq 0}\;\;\text{with}\;\;|\beta|=r\}.$$ Fix $\beta_0$ such that $|\beta_0|=r$ and $c_{\beta_0}>0$.

\begin{lemma}\label{LemmaPhiFA} For any clopen set $A\subsetneq X$ we have that \begin{equation}\label{EqnF(A)=Sum_Measures} \varphi(\F(A)) = \sum_{\alpha \in
\mathbb Z^k_{\geq 0}} c_\alpha \mu_\alpha(A^{|\alpha|}).
\end{equation}
\end{lemma}
\begin{proof}
Recall that for any $g\in G$ we have that $$\varphi(\fix_\Omega(g))=c_\infty
\chi_\reg(g) + \sum_{\alpha \in
\mathbb Z^k_{\geq 0}} c_\alpha \mu_\alpha(\fix_{X^{|\alpha|}}(g)).$$
Recall that by Lemma \ref{LemmaProperties2} for any clopen set $A\subsetneq X$, the orbits of $\G(A)$ on $X\setminus A$ are infinite. In particular, $|\G_n(A)|\to\infty$ as $n\to\infty$.  Using Corollary \ref{corollaryLimitFixedPoints} and the Monotone Convergence
Theorem,  we obtain that for a given clopen set $A\subsetneq X$,
\begin{eqnarray*} \varphi(\F(A)) & = &
\lim_{n\to\infty}\frac{1}{|\G_n(A)|}\sum\limits_{g\in
\G_n(A)}\varphi(\fix_{\Omega}(g)) \\
& = & \lim_{n\to\infty}\frac{1}{|\G_n(A)|}  \sum\limits_{g\in
\G_n(A)} \bigg( c_\infty
\chi_\reg(g) + \sum_{\alpha \in
\mathbb Z^k_{\geq 0}} c_\alpha \mu_\alpha(\fix_{X^{|\alpha|}}(g)) \bigg) \\
 & = & \sum_{\alpha \in
\mathbb Z^k_{\geq 0}} \lim_{n\to\infty}\frac{1}{|\G_n(A)|}
\sum\limits_{g\in \G_n(A)}
 c_\alpha \mu_\alpha(\fix_{X^{|\alpha|}}(g)) \\
 & = & \sum_{\alpha \in
\mathbb Z^k_{\geq 0}} c_\alpha \mu_\alpha(A^{|\alpha|}).
\end{eqnarray*}
\end{proof}
Fix a  sequence of nested clopen partitions $\{\Xi_n\}_{n=1}^\infty$ of $X$
with $\max\limits_{C\in \Xi_n}\mathrm{diam}(C)\to 0.$ Denote by $\Xi_{p,n}$
the collection of $p$-tuples $\{C_1,\ldots,C_p\}$ of pairwise distinct sets $C_i\in \Xi_n$,
$i=1,\ldots, p$.  For $\mathcal{ C} = \{C_1,\ldots, C_p\}\in \Xi_{p,n}$ denote by
$\cup \mathcal{ C}$ the union $C_1\cup \cdots \cup C_p$. Notice that $\Xi_{p,n}$ is empty whenever $|\Xi_n|<p$. Fix $N$ large enough so that for any $p\leq r$, $n\geq N$ and any $\mathcal C,\mathcal D\in \Xi_{p,n}$ we have $(\cup \mathcal{ C})\cup (\cup \mathcal{ D})\neq X$. Thus, by Lemma \ref{LemmaProperties1} $$\mathcal F((\cup \mathcal{ C})\cap (\cup \mathcal{ D}))=\mathcal F(\cup \mathcal{ C})\cap \mathcal F(\cup \mathcal{ D}).$$
%
%

\begin{lemma}\label{LmPhiUnF} For any $n\geq N$ and any $p\leq r$ we have that \begin{eqnarray*} \varphi\bigg( \bigcup_{\mathcal{C}\in \Xi_{p,n}} \F(\cup \mathcal{C}) \bigg) & = &\sum_{\alpha \in
\mathbb Z^k_{\geq 0}} c_\alpha \mu_\alpha\bigg( \bigcup_{\mathcal{C}\in \Xi_{p,n}} (\cup\mathcal{C})^{|\alpha|} \bigg).
\end{eqnarray*}
\end{lemma}
\begin{proof} Note that for any sets $A,B\subset X$, \begin{equation}\label{EquationAIntersectB}(A\cap B)^{|\alpha|} = A^{|\alpha|}\cap B^{|\alpha|}.\end{equation}
For any $n\geq N$ applying the inclusion-exclusion principle (IEP) to the union of sets in $\Xi_{p,n}$, we obtain that
\[
\begin{array}{lcl}
 \varphi\bigg( \bigcup_{\mathcal{C}\in \Xi_{p,n}} \F(\cup \mathcal{C}) \bigg) & = &
\sum_{\emptyset \neq J\subset \Xi_{p,n}} (-1)^{|J|-1}\varphi\bigg(\bigcap_{\mathcal{C}\in J} \F(\cup \mathcal{C}) \bigg)  \\
\mbox{by Lemma \ref{LemmaProperties1}}& = &
 \sum_{\emptyset \neq J\subset \Xi_{p,n}} (-1)^{|J|-1}\varphi\bigg(\F\bigg( \bigcap_{\mathcal{C}\in J} \cup \mathcal{C}\bigg ) \bigg)  \\
\mbox{by Equation (\ref{EqnF(A)=Sum_Measures})} & = & \sum_{\alpha \in
\mathbb Z^k_{\geq 0}} c_\alpha \sum_{\emptyset \neq J\subset \Xi_{p,n}} (-1)^{|J|-1}  \mu_\alpha\bigg(\bigg( \bigcap_{\mathcal{C}\in J} \cup \mathcal{C}\bigg )^{|\alpha|}\bigg)\\
 \mbox{by Equation (\ref{EquationAIntersectB})}& = & \sum_{\alpha \in
\mathbb Z^k_{\geq 0}} c_\alpha \sum_{\emptyset \neq J\subset \Xi_{p,n}} (-1)^{|J|-1}  \mu_\alpha\bigg(\bigcap_{\mathcal{C}\in J} (\cup \mathcal{C})^{|\alpha|}\bigg)\\
\mbox{by the IEP} & = & \sum_{\alpha \in
\mathbb Z^k_{\geq 0}} c_\alpha \mu_\alpha\bigg( \bigcup_{\mathcal{C}\in \Xi_{p,n}} (\cup\mathcal{C})^{|\alpha|} \bigg).
\end{array}
\]
\end{proof}

\noindent Since the union $\bigcup\limits_{\mathcal{ C}\in\Xi_{r,n}}(\cup \mathcal{ C})^{r}$ covers the set $X^r$, we immediately obtain the following result.  \begin{corollary}\label{CoPhiLowerEstimates} For any $n\geq N$ we have
\begin{equation}\label{EqPhiLowerEstimates}\varphi\left( \bigcup_{\mathcal{ C}\in\Xi_{r,n}} \F(\cup\mathcal{ C}) \right)\geq c_{\beta_0}>0.\end{equation}\end{corollary}

For $p\in\mathbb Z_{\geq 0}$ denote by $X_p$ the collection of all finite subsets of $X$ of cardinality at most $p$.  Given $\bar x = \{x_1,\ldots,x_m \}\in X_p$, denote by $\mathcal{ C}_n(\bar x) = \{C_1,\ldots,C_q\}$ the element of $\Xi_{q,n}$ with $x_i\in \cup\mathcal{ C}_n(\bar x)$ for $i=1,\ldots,m$ with $q$ being the least possible. Thus, for every $1\leq j\leq q$ there exists at least one $1\leq i\leq m$ such that $x_i\in C_j$. Notice that $q\leq m \leq p$ and $q=m$ for sufficiently large $n$.

\begin{lemma}\label{LemmaIntersFC} For any $p\geq 0$ we have  \begin{eqnarray}\label{EqDescriptionGamma}\bigcap_{n\geq N}\bigcup_{\mathcal C\in \Xi_{p,n}} \F(\cup \mathcal C) = \bigcup\limits_{\bar x\in X_p}\mathcal F(\bar x).\end{eqnarray}\end{lemma}
\begin{proof}
 Given $\bar x\in X_p$ and $n\geq N$, by Remark \ref{RemarkFFamilyIncreasing},  one has $$\mathcal F(\bar x)\subset\F(\cup \mathcal{ C}_n(\bar x))\subset\bigcup_{\mathcal C\in \Xi_{p,n}} \F(\cup \mathcal C).$$ Conversely, suppose $H\in  \F(\cup \mathcal C^{(n)})$ for every $n \geq N$, where $\mathcal{C}^{(n)}\in \Xi_{p,n}$. Set $B_n=\bigcap_{j=1}^n \cup \mathcal C^{(j)}.$ Using Lemma \ref{LemmaProperties1}, we obtain that $H\in  \F(B_n)$ for every $n \geq N$. The sequence $\{B_n\}$ is decreasing and converging to a $q$-element subset $\bar x=\{x_1,\ldots,x_q\}$ of $X$ for some $q\leq p$.  Using Lemma \ref{LemmaDecrFam} we obtain that $H\in \mathcal F(\bar x)$. This completes the proof.\end{proof}
%
%
Note that by Remark \ref{RemarkFFamilyIncreasing} the  sets $\left \{\bigcup_{\mathcal C\in \Xi_{r,n}} \F(\cup \mathcal C)\right \},n\geq N,$ form a decreasing family.
It follows from Corollary \ref{CoPhiLowerEstimates} and Lemma \ref{LemmaIntersFC} that
\begin{equation}\label{EqLoweBound}\varphi\left( \bigcup_{\bar x\in X_r}\F(\bar x)\right) \geq c_{\beta_0}>0.\end{equation}
Therefore, by ergodicity of $\varphi$, we obtain that the measure $\varphi$ is supported by the set $ \bigcup_{\bar x\in X_r}\F(\bar x)$.   We can further assume that $r$ is the least number for which Equation (\ref{EqLoweBound}) still holds, which implies that the measure $\varphi$ is supported by the set $\bigcup_{\bar x\in W_r}\F(\bar x)$, where $W_r$ consists of all $r$-element subsets of $X_r$. That is,
\begin{equation}\label{EqLoweBoundW}\varphi\left( \bigcup_{\bar x\in W_r}\F(\bar x)\right)=1.\end{equation}
 Denote by $Z_r$ be the subset of $W_r$ consisting of sets $\bar x$ such that the orbits $\{Gy : y\in \bar x\}$ are pairwise distinct.

\begin{lemma}\label{LemmaPhiXr-Zr}  We have that
$$\varphi\left(\bigcup_{\bar x\in W_r\setminus Z_r}\F(\bar x)\right) = 0.$$
\end{lemma}
\begin{proof} First, we construct a collection $\mathfrak C$ of measurable subsets of $W_r\setminus Z_r$ covering $W_r\setminus Z_r$ such that for any $C\in \mathfrak C$ there exists a sequence $h_j\in G$ for which the sets $h_jC$, $j\geq 1$,  are pairwise disjoint. Given a clopen set $A$ and an element $g\in G$ with $g(A)\cap A=\varnothing$, set $$C_{A,g}=\{\bar x\in W_r:\exists y,z\in \bar x,\; y\in A,\;z=gy\;\text{and}\;t\notin A\cup gA\;\text{for all}\;t\in \bar x\setminus \{y,z\}\}.$$ Given $A,g$ as above, pick a sequence $h_j\in G$ with $\supp(h_ih_j^{-1})=A$ for each $i\neq j$. Then the sets $h_jC_{A,g},j\geq 1,$ are pairwise disjoint. The collection $\mathfrak C$ will consist of all set of the form $C_{A,g}$. Notice that the collection $\mathfrak C$ covers $W_r\setminus Z_r$.

Given $C\in \mathfrak C$ and a sequence of group elements $\{h_j\}$ as above, for any $i\neq j$ we have that
$$\left(\bigcup_{\bar x\in h_iC} \F(\bar x) \right) \cap \left(\bigcup_{\bar x\in h_jC} \F(\bar x) \right) \subset \left(\bigcup_{\bar x\in X_{r-1}} \F(\bar x)\right),$$ which is a $\varphi$-null set. Since
 $$\varphi\left(\bigcup_{\bar x\in h_jC} \F(\bar x)\right)=\varphi\left(\bigcup_{\bar x\in C} \F(\bar x)\right)$$ for any $j$ and $\varphi$ is a probability measure, we obtain that $\varphi(\bigcup_{\bar x\in C} \F(\bar x)) = 0$. Thus, the set $\bigcup_{\bar x\in W_r\setminus Z_r}\F(\bar x)$ can be covered by a countable collection of $\varphi$-null sets. This establishes the result.
\end{proof}

Combing Equation \eqref{EqLoweBoundW}, Lemma \ref{LemmaPhiXr-Zr} and Lemma \ref{LemmaProperties3}, we conclude that the measure $\varphi$ is supported by the $G$-invariant set  $\mathcal S= \bigcup_{l \geq 0}\{\G(\bar x):\bar x\in X_l\}$. Observe that $\G({\bar x})$ is the stabilizer subgroup of $\bar x$, that is, $g\in \G(\bar x)$ if and only if $g(x_i) = x_i$ for every $i$.
  It follows that for every $\alpha\in \mathbb Z^k_{\geq 0}$ the measure $\varphi_\alpha$ is also supported by the set $\mathcal S$. By Lemma \ref{LemmaPhiFA} for every clopen set $A\subsetneq X$ we have
$$\varphi(\F(A))=\sum_{\alpha\in \mathbb Z_{\geq 0}^k} c_\alpha\varphi_\alpha(\F(A)).$$  Furthermore, it is straightforward to check that the collection of sets
 $$\mathcal A=\{\mathcal F(A):A\subset X\;\text{is clopen and such that }\mu_i(A)<1/3\mbox{ for every }i\}$$ separates points of $\mathcal S$, \ie for any pair of elements of $\mathcal S$ there exists $F\in\mathcal A$ containing exactly one of the elements from the pair. Thus, the family $\mathcal A$ generates the Borel $\sigma$-algebra on $\mathcal S$. Moreover, the condition that $\mu_i(A)<1/3$ for every $i$ and every $A\in \mathcal A$ ensures that the family $\mathcal A$  is closed under finite intersections, see  Lemma \ref{LemmaProperties1}. Therefore, by \cite[Corollary 1.6.3]{Cohn:MeasureTheory},  we conclude that $$\varphi \equiv \sum_{\alpha\in \mathbb Z^k_{\geq 0 }}c_\alpha\varphi_\alpha.$$ By ergodicity of $\varphi$, we obtain that $\varphi = \varphi_\alpha$ for some $\alpha\in \mathbb Z^k_{\geq 0}$, which completes the proof of Theorem \ref{TheoremMainIntro}.

%
%

\bibliographystyle{abbrv}
\bibliography{bibliography}

\end{document}